\documentclass[reqno, 11pt, a4paper]{amsart}

\usepackage{latexsym,ifthen,xspace}
\usepackage{amsmath,amssymb,amsthm}
\usepackage{enumerate, calc}
\usepackage[colorlinks=true, pdfstartview=FitV, linkcolor=blue,
  citecolor=blue, urlcolor=blue, pagebackref=false]{hyperref}
\usepackage[text={33pc,605pt},centering]{geometry}    
\usepackage{graphicx,color}

\newtheorem{theorem}{Theorem}[section]
\newtheorem{lemma}[theorem]{Lemma}
\newtheorem{proposition}[theorem]{Proposition}
\newtheorem{assumption}[theorem]{Assumption}
\newtheorem{corollary}[theorem]{Corollary}

\theoremstyle{definition}
\newtheorem{definition}[theorem]{Definition}

\theoremstyle{remark}
\newtheorem{remark}[theorem]{Remark}

\numberwithin{equation}{section}

\DeclareMathAlphabet{\mathsl}{OT1}{cmss}{m}{sl}
\SetMathAlphabet{\mathsl}{bold}{OT1}{cmss}{bx}{sl}

\newcommand{\al}{\ensuremath{\alpha}}

\newcommand{\de}{\ensuremath{\delta}}

\renewcommand{\th}{\ensuremath{\theta}}

\newcommand{\ka}{\ensuremath{\kappa}}
\newcommand{\la}{\ensuremath{\lambda}}

\newcommand{\si}{\ensuremath{\sigma}}

\newcommand{\om}{\ensuremath{\omega}}

\newcommand{\De}{\ensuremath{\Delta}}

\newcommand{\La}{\ensuremath{\Lambda}}
\newcommand{\Si}{\ensuremath{\Sigma}}

\newcommand{\Om}{\ensuremath{\Omega}}
\newcommand{\cA}{\ensuremath{\mathcal A}}
\newcommand{\cB}{\ensuremath{\mathcal B}}

\newcommand{\cE}{\ensuremath{\mathcal E}}
\newcommand{\cF}{\ensuremath{\mathcal F}}
\newcommand{\cG}{\ensuremath{\mathcal G}}
\newcommand{\cH}{\ensuremath{\mathcal H}}

\newcommand{\cL}{\ensuremath{\mathcal L}}

\newcommand{\cR}{\ensuremath{\mathcal R}}

\newcommand{\bbE}{\ensuremath{\mathbb E}}

\newcommand{\bbN}{\ensuremath{\mathbb N}} 
 
\newcommand{\bbP}{\ensuremath{\mathbb P}} 
 
\newcommand{\bbR}{\ensuremath{\mathbb R}}

\newcommand{\bbZ}{\ensuremath{\mathbb Z}} 
\definecolor {orange} {rgb} {0.569, 0.259, 0.0}

\newcommand{\me}{\ensuremath{\mathrm{e}}}

\let\norm\undefined
\newcommand{\norm}[3]{%
  \ensuremath{%
    \big\lVert
      #1
    \big\rVert_{\raisebox{-.0ex}{$\scriptstyle \ell^{\raisebox{.2ex}{$\scriptscriptstyle #2$}} (#3)$}}
  }
}

\DeclareMathOperator{\supp}{\mathrm{supp}}

\DeclareMathOperator{\esssup}{\mathrm{ess}\,\mathrm{sup}}

\newcommand{\indic}{\mathbbm{1}}
\newcommand{\grad}{\nabla}
\newcommand{\eps}{\epsilon}

\usepackage{mathtools}
\DeclarePairedDelimiter\ceil{\lceil}{\rceil}
\DeclarePairedDelimiter\floor{\lfloor}{\rfloor}
\usepackage{commath}
\usepackage{color}
\usepackage{cite}
\usepackage{graphicx,bbm,enumerate,amsthm}

\let\norm\undefined
\newcommand{\norm}[1]{%
  \ensuremath{%
    \big\lVert
      #1
    \big\rVert}
}

\begin{document}

\title[Heat Kernel Estimates for Symmetric Diffusions]{Off-Diagonal Heat Kernel Estimates for Symmetric Diffusions in a Degenerate Ergodic Environment}


\author{Peter A.\ Taylor}
\address{University of Cambridge}
\curraddr{Centre for Mathematical Sciences, Wilberforce Road, Cambridge CB3 0WB}
\email{pat47@cam.ac.uk, peter.taylor876@gmail.com}
\thanks{}



\date{\today}

\dedicatory{}

\begin{abstract}
We study a symmetric diffusion process on $\bbR^d$, $d\geq 2$, in divergence form in a stationary and ergodic random environment. The coefficients are assumed to be degenerate and unbounded but satisfy a moment condition. We derive upper off-diagonal estimates on the heat kernel of this process for general speed measure. Lower off-diagonal estimates are also shown for a natural choice of speed measure, under an additional mixing assumption on the environment. Using these estimates, a scaling limit for the Green's function is proven.
\end{abstract}

\maketitle

\section{Introduction}

We study a diffusion process on $\bbR^d$, formally associated with the following generator
\begin{equation}
\label{eq:generator}
\cL^\om u(x)=\frac{1}{\th^\om(x)}\nabla\cdot(a^\om(x)\nabla u(x)),\quad x\in\bbR^d,
\end{equation}
where the random field $\{a^\om(x)\}_{x\in\bbR^d}$ is a symmetric $d$-dimensional matrix for each $x\in \bbR^d$, and $\th^\om$ is a positive speed measure which may also depend on the random environment $\om$. Firstly, we set out the precise assumptions on the random environment. Let $\big(\Om, \cG, \bbP,\{\tau_x\}_{x\in\bbR^d}\big)$ be a probability space together with a measurable group of translations. $\bbE$ will denote the expectation under this probability measure. To construct the random field let $a:\Om\to \bbR^{d\times d}$ be a $\cG$-measurable random variable and define $a^\om(x):=a(\tau_x\om)$. The speed measure is defined similarly, take a $\cG$-measurable random variable $\th: \Om \to (0,\infty)$ and let $\th^\om(x):=\th(\tau_x \om)$. We refer to this function as the speed measure because the process with general $\th^\om$ can be obtained from the process with $\th^\om\equiv 1$ via a time-change. As made precise in the following, we assume throughout that the random environment is stationary, ergodic and satisfies a non-uniform ellipticity condition.

\begin{assumption}
\label{ass:ergodicity}
The probability space satisfies:
\begin{enumerate}[(i)]
\item $\bbP(\tau_x A)=\bbP(A)$ for all $A\in\cG$ and any $x\in\bbR^d$.
\item If $\tau_x A=A$ for all $x\in\bbR^d$ then $\bbP(A)\in\{0,1\}$.
\item The mapping $(x,\om)\mapsto \tau_x\,\om$ is $\cB(\bbR^d)\otimes\cG$ measurable.
\end{enumerate}
Furthermore for each $x\in\bbR^d$, $a^\om(x)$ is symmetric and there exist positive, $\cG$-measurable $\la,\, \La:\Om\to(0,\infty)$ such that
for $\bbP$-a.e.~$\om\in\Om$ and all $\xi\in\bbR^d$, $x\in\bbR^d$,
\begin{equation}
\label{eq:ellipticity}
\la(\tau_x\om)\abs{\xi}^2\leq \xi \cdot (a^\om(x)\xi) \leq \La(\tau_x\om)\abs{\xi}^2.
\end{equation}
Also, defining $\La^\om(x):=\La(\tau_x\om)$ and $\la^\om(x):=\la(\tau_x\om)$ for $x\in\bbR^d$, assume that $\bbP$-a.s.
\begin{equation}
\La^\om,\,(\la^\om)^{-1},\,\th^\om,\,(\th^\om)^{-1} \in L^\infty_{\textrm{loc}}(\bbR^d).
\end{equation}
\end{assumption}
The final assumption of local boundedness will allow us to pass from estimates on the semigroup of the diffusion process to pointwise bounds on the heat kernel. Rather than assuming these functions are uniformly bounded, we work with moment conditions given in terms of the following, for $p,q,r\in(0,\infty]$ define
\begin{align}
\nonumber
M_1(p,q,r)&:=\bbE[\th^\om(0)^r]+\bbE[\la^\om(0)^{-q}]+\bbE[\La^\om(0)^p\th^\om(0)^{1-p}],\\
\label{eq:moment terms defn}
M_2(p,q)&:=\bbE[\la^\om(0)^{-q}]+\bbE[\La^\om(0)^p].
\end{align}
By the ergodic theorem, these conditions together with Assumption~\ref{ass:ergodicity} allow us to control average values of the functions on large balls. For instance, denoting $B(x,r)$ the closed Euclidean ball of radius $r$ centred at $x$, $\bar{\La}_p:=\bbE[\La^\om(0)^p]$ and $\bar{\la}_q:=\bbE[\la^\om(0)^q]$, then $M_2(p,q)<\infty$ implies that for $\bbP$-a.e.~$\om$, there exists $N_1^\om(x)>0$ such that for all $r\geq N_1^\om$,
\begin{equation}
\label{eq:ergodic control}
\frac{1}{\abs{B(x,r)}}\int_{B(x,r)}\La^\om(u)^p\,du<2\,\bar{\La}_p,\qquad \frac{1}{\abs{B(x,r)}}\int_{B(x,r)}\la^\om(u)^q\,du<2\,\bar{\la}_q.
\end{equation}

In the uniformly elliptic case, where $\La^\om(x)$ and $\la^\om(x)$ are bounded above and below respectively, uniformly in $\om$, the model we are considering has been extensively studied. A quenched invariance principle is established in \cite{PV81} for differentiable, periodic coefficients. Further results for smooth, periodic, uniformly elliptic coefficients are given in \cite{Osa83}. The quenched invariance principle was extended to a random environment with a uniformly elliptic symmetric part and differentiable skew-symmetric part satisfying a growth condition in \cite{FK97}. Outside the uniformly elliptic regime and more recently, \cite{BM15} proved this homogenization result for operators taking a specific, periodic form, with measurable and locally integrable coefficients. Without assuming differentiability of the random field, some work is required to construct the process associated with \eqref{eq:generator} in a general ergodic environment. The diffusion is constructed using the theory of Dirichlet forms, with the corresponding form being
\begin{equation}
\label{eq:dirichlet form}
\cE^\om(u,v):=\sum_{i,j=1}^d \int_{\bbR^d}a_{ij}^\om(x)\partial_i u(x)\partial_j v(x)\,dx,
\end{equation}
for $f,\, g$ in a proper class of functions $\cF^\theta\subset L^2(\bbR^d,\th^\om\, dx)$, defined precisely in Section~\ref{section:davies}. The construction of a diffusion process $(X_t^\th)_{t\geq 0}$  
associated to \eqref{eq:generator} is a recent result of \cite{CD16}. This is done under Assumption~\ref{ass:ergodicity} together with the moment condition $M_2(p,q)<\infty$ for some $p,\,q\in(1,\infty]$ satisfying $\frac{1}{p}+\frac{1}{q}<\frac{2}{d}$. The main result in \cite[Theorem 1.1]{CD16} is the quenched invariance principle, that is for $\bbP$-a.e.~$\om$ the law of the process $\big(\frac{1}{n} X_{n^2t}\big)_{t\geq 0}$ on $C([0,\infty),\bbR^d)$ converges weakly as $n\to \infty$ to that of a Brownian motion. This is first proven for $\th^\om\equiv 1$ and then for general speed measure satisfying $\bbE[\th^\om(0)]<\infty$ and $\bbE[\th^\om(0)^{-1}]<\infty$, after showing that the general speed process can be obtained via a time change.

Regarding the heat kernel of the operator $\cL^\om$, it is also shown therein that the semigroup $P_t$ of the above diffusion process has a transition kernel $p_\th^\om(t,x,y)$ with respect to $\th^\om(x)\, dx$, furthermore this is jointly continuous in $x$ and $y$. Explicitly, for continuous, bounded $f:\bbR^d\to\bbR$,
\begin{equation}
\label{eq:semigroup relation}
P_t f(x)=\int_{y\in\bbR^d}f(y)p_\th^\om(t,x,y)\th^\om(y)\, dy,\qquad\forall\, x\in\bbR^d,\, t>0.
\end{equation}

A second, stronger result that has recently been established under Assumption~\ref{ass:ergodicity} and moment condition $M_1(p,q,r)<\infty$ for some $p,q,r\in(1,\infty]$ satisfying $\frac{1}{r}+\frac{1}{q}+\frac{1}{p-1}\frac{r-1}{r}<\frac{2}{d}$ is the quenched local central limit theorem \cite[Theorem 1.1]{CD15}. This states that the rescaled transition kernel $p_\th^\om(n^2t,0,nx)$ converges as $n\to\infty$ to the heat kernel of a Brownian motion $k^\Si_t(0,x)$ with some deterministic, positive definite covariance matrix $\Si$ implicitly depending on the law $\bbP$. Namely, for $t>0$ and $x,\, y \in\bbR^d$,
\begin{equation}
\label{eq:gaussian kernel defn}
k^\Si_t(x,y):= \frac{1}{\sqrt{(2\pi t)^d \det \Si}}\exp\Big(-\frac{(y-x)\cdot \Si^{-1} (y-x)}{2 t}\Big).
\end{equation}
The convergence is uniform on compact sets in $t$ and $x$ and the key step is to apply a parabolic Harnack inequality to obtain H\"older regularity of the heat kernel; this is achieved via Moser iteration which will also play an important role in our analysis. Many of the techniques take inspiration from the random conductance model (RCM) setting, cf.~\cite{harnack, AT19, ACS20, BS20} for recent RCM local limit theorems in a degenerate, ergodic setting. The diffusion studied in this paper is a continuum analogue of that model, where a random walk moves on a lattice, usually $\bbZ^d$ equipped with nearest-neighbour edges. Importantly, the RCM literature indicates that moment conditions are indeed necessary for a general ergodic environment, for instance \cite{harnack} proves a local limit theorem under a moment assumption equivalent to $M_2(p,q)<\infty$ for $\frac{1}{p}+\frac{1}{q}<\frac{2}{d}$ and shows that this condition is optimal for the canonical choice of speed measure (known as the constant speed random walk). Another recent result under moment conditions is a Liouville theorem for the elliptic equation associated to \eqref{eq:generator} in \cite{BFO18}, cf.~also \cite{BCF19} for a related result on the parabolic equation associated to a time-dynamic, uniformly elliptic version of \eqref{eq:generator}. Local boundedness and a Harnack inequality for solutions to the elliptic equation were recently proven in \cite{BS21} under moment conditions.

A local limit theorem quantifies the limiting behaviour of the heat kernel and is known to provide near-diagonal estimates on the kernel prior to rescaling -- see Proposition~\ref{prop:ndle}. In this paper our aim is to derive full Gaussian estimates on the heat kernel $p_\th^\om(t,x,y)$ for all $x$ and $y$, also known as off-diagonal estimates. For general speed measure, it is known that these bounds should be governed by the intrinsic metric, cf.~\cite{CKS87, Dav87, Str88}. In the random environment setting, this is a metric on $\bbR^d$ dependent on $a^\om$ and $\th^\om$, defined as
\begin{equation}
\nonumber
d^\om_\th(x,y):=\sup\Big\lbrace \phi(y)-\phi(x) \, :\, \phi \in C^1(\bbR^d),\, h^\om(\phi)^2= \sup_{z\in\bbR^d}\frac{(\grad\phi\cdot a^\om\grad\phi)(z)}{\th^\om(z)}\leq 1\Big\rbrace.
\end{equation}
Outside of the uniformly elliptic case it is clear that the above is not in general comparable to the Euclidean metric, which we denote $d(\cdot\,,\cdot)$. A natural follow-up question to this work would be to find the minimal conditions on an ergodic environment for which these two metrics are comparable. However here we require some regularity of the intrinsic metric in order to derive off-diagonal heat kernel estimates in terms of it. Specifically we must show it is strictly local, meaning it generates the Euclidean topology on $\bbR^d$. We therefore make the following additional assumption.

\begin{assumption}[Continuity of the Environment]
\label{ass:continuity}
For $\bbP$-a.e.~$\om$, the functions $a^\om:\bbR^d\to\bbR^{d\times d}$ and $\th^\om:\bbR^d\to(0,\infty)$ are continuous.
\end{assumption}

Our first main result is an upper off-diagonal heat kernel estimate for the symmetric diffusion process with general speed measure in an ergodic, degenerate environment, and is proven in Section~\ref{section:davies}.
\begin{theorem}
\label{thm:off-diagonal estimate}
Suppose Assumption~\ref{ass:ergodicity} and Assumption~\ref{ass:continuity} hold. Let $d\geq2$ and assume $M_1(p,q,r)<\infty$ for some $p,q,r\in(1,\infty]$ satisfying $\frac{1}{r}+\frac{1}{q}+\frac{1}{p-1}\frac{r-1}{r}<\frac{2}{d}$. Then for $\bbP$-a.e.~$\om\in\Om$ and every $x\in\bbR^d$, there exist $N_2^\om(x)>0$, $c_1(d,p,q,r)>0$ and $\gamma(d,p,q,r)>0$ such that the following holds for all $y\in\bbR^d$ and $\sqrt{t}>N_2^\om(x)$,
\begin{equation}
\label{eq:off-diagonal estimate}
p_\th^\om(t,x,y)\leq c_1\, t^{-\frac{d}{2}}\Big(1+\frac{d(x,y)}{\sqrt{t}}\Big)^\gamma \exp\Big(-\frac{d_\th^\om(x,y)^2}{8t}\Big).
\end{equation}
\end{theorem}
\begin{remark}
\begin{enumerate}[(i)]
\item In the `constant speed' setting of $\th^\om\equiv \La^\om$ (and more generally whenever $\th^\om\geq c\,\La^\om$), we obtain off-diagonal heat kernel estimates in terms of the Euclidean metric $d$, without the need for Assumption~\ref{ass:continuity}. We get a full Gaussian upper estimate here because the polynomial prefactor in \eqref{eq:off-diagonal estimate} can be absorbed into the exponential when the two metrics are comparable. See Corollary~\ref{cor:euclidean off-diagonal} for the precise statement. 

\item We restrict to $d\geq 2$ because the technique for deriving the maximal inequality in Section~\ref{subsection:cauchy problem} breaks down in $d=1$ due to issues with the Sobolev inequality (Lemma~\ref{lemma:sobolev}). Similarly, the lower estimate (Theorem~\ref{thm:lower estimate}) relies on the parabolic Harnack inequality in \cite{CD15} which is also derived using Moser iteration for $d\geq 2$ only.
\end{enumerate}
\end{remark}
To prove the above estimate we use Davies' perturbation method, a technique for deriving upper off-diagonal estimates, well-established in the elliptic and parabolic equations literature for uniformly elliptic operators cf.~\cite{Dav87, Dav90, CKS87, Str88, GT12}. The idea also translates to heat kernels on graphs \cite{Dav93, Del99} and recently the RCM in a degenerate, ergodic environment \cite{ADS16, ADS19}. The first step of Davies' method is to consider the Cauchy problem associated to the perturbed operator $\cL_\psi^\om:=\me^\psi\cL^\om\me^{-\psi}$ where $\psi$ is an arbitrary test function, and use a maximal inequality to bound the fundamental solution. In \cite{Zhi13}, off-diagonal estimates are derived for solutions of a parabolic equation in a uniformly elliptic setting with degenerate, locally integrable weight; this was useful inspiration for the Cauchy problem we consider in Section~\ref{subsection:cauchy problem}. To derive the maximal inequality we use a Moser iteration scheme adapted to the perturbed operator, similar to the method used to derive the parabolic Harnack inequality for the original operator $\cL^\om$ in \cite{CD15}. Ergodic theory plays a key role here in controlling constants which depend on the random environment. Moser iteration has previously been applied to prove the corresponding RCM results -- the quenched invariance principle in \cite{QFCLT}, the Harnack inequality in \cite{harnack} and off-diagonal estimates in \cite{ADS16, ADS19}.

The second part of the argument is to optimise over the test function $\psi$. In the uniformly elliptic case this is straightforward as one can work with the Euclidean metric, however in our general setting of degenerate coefficient matrix and speed measure the off-diagonal estimate is governed by the intrinsic metric defined above. Utilising a test function related to this metric requires certain regularity properties, for instance that it generates the Euclidean topology on $\bbR^d$. In Section~\ref{subsection:intrinsic} we first relate the intrinsic metric to a Riemannian metric, then apply a recent result from geometric analysis \cite{Bur15} to prove the necessary regularity properties under Assumption~\ref{ass:continuity}.

As a counterpart to the preceding upper estimate, we also present a lower off-diagonal estimate for the heat kernel, in the `constant speed' case of $\th^\om\equiv \La^\om$. Whilst Assumption~\ref{ass:continuity} for regularity of the intrinsic metric is no longer required, we need stronger control on the environment than given in Assumption~\ref{ass:ergodicity}. In particular, a decorrelation assumption is necessary for the proof and we assume finite-range dependence of the environment.
\begin{assumption}
\label{ass:finite range}
Suppose there exists a positive constant $\cR>0$ such that for all $x\in\bbR^d$ and $\bbP$-a.e.~$\om$, $\tau_x \om$ is independent of $\{\tau_y\om\, : \, y \in B(x,\cR)^c \}$.
\end{assumption}
To prove the lower estimate we adapt the established chaining argument to the diffusion in a degenerate random environment, the method originated in \cite{FS86} using the ideas of Nash. It was adapted to the weighted graph setting in \cite{Del99}, to random walks on percolation clusters in \cite{Bar04}, and it was recently applied to the RCM \cite{AH20}. The strategy is to repeatedly apply lower near-diagonal estimates, derived from the parabolic Harnack inequality established in \cite{CD16}, along a sequence of balls. The form of the constant in the Harnack inequality means that averages of the functions $\la^\om(\cdot)$, $\La^\om(\cdot)$ on balls with varying centre-points must be controlled simultaneously to derive the lower off-diagonal estimate. Something stronger than the classical ergodic theorem is required to do this, so given Assumption~\ref{ass:finite range} we establish a specific form of concentration inequality (Proposition~\ref{prop:moment bound}) for this purpose. By an argument similar to \cite{AH20} this inequality is then used to control the environment-dependent terms arising from the Harnack inequality, see Proposition~\ref{prop:ergodic terms bnd}. The statement is given below and proven in Section~\ref{section:lower estimate}.

\begin{theorem}
\label{thm:lower estimate}
Suppose $d\geq2$ and Assumptions~\ref{ass:ergodicity} and \ref{ass:finite range} hold. There exist $p_0,\, q_0\in (1,\infty)$ such that if $M_2(p_0,q_0)<\infty$ then for $\bbP$-a.e.~$\om$ and every $x\in \bbR^d$, there exist $c_i(d)>0$ and a random constant $N_3^\om(x)>0$ satisfying
\begin{equation}
\label{eq:N decay}
\bbP(N_3^\om(x)> n)\leq c_2\, n^{-\alpha}\qquad \forall\, n>0,
\end{equation}
for some $\al>d(d-1)-2$, such that the following holds. For all $y \in \bbR^d$ and $t \geq N_3^\om(x)\big(1\vee d(x,y)\big)$,
\begin{equation}
\label{eq:lower estimate statement}
p_\La^\om(t,x,y)\geq c_3 \, t^{-d/2}\exp\Big(-c_4\frac{d(x,y)^2}{t}\Big).
\end{equation}

\end{theorem}
\begin{remark}
\begin{enumerate}[(i)]
\item In \cite{AH20}, three other assumptions such as an FKG inequality or a spectral gap inequality are offered as alternatives to finite-range dependence. Some of these are specific to the discrete setting and Assumption~\ref{ass:finite range} is the most natural for our context, but it may be possible to replace it with other similar conditions.
\item We state the above only for $\th^\om\equiv \La^\om$ because for general speed measure the intrinsic metric is not necessarily comparable to the Euclidean metric which is used for the chaining argument. It may be possible to adapt the argument to general speed however and it is unclear whether this would require further assumptions in order to compare the two metrics.
\end{enumerate}
\end{remark}
Our final result is a scaling limit for the Green's function of the diffusion process, defined as
$$g^\om(x,y):=\int_0^\infty p_\La^\om(t,x,y)\, dt.$$
As already noted, the diffusion with general speed measure may be obtained from the process with speed measure $\th^\om\equiv \Lambda^\om$ via a time change \cite[Theorem~2.4]{CD15}. Therefore the Green's function, which exists in dimension $d\geq 3$ due to the upper off-diagonal heat kernel estimate above, is independent of the speed measure $\th^\om$. Applying Theorem~\ref{thm:off-diagonal estimate} together with a long-range bound obtained in Section~\ref{section:green kernel}, we obtain sufficient bounds to apply the local limit theorem \cite[Theorem~1.1]{CD15} and show that an appropriately rescaled version of the Green's function converges to that of a Brownian motion,
\begin{equation}
g_{\textrm{BM}}(x,y):=\int_0^\infty k^\Si_t(x,y)\, dt.
\end{equation}
\begin{theorem}
\label{thm:green kernel}
Let $d\geq 3$ and suppose Assumption~\ref{ass:ergodicity} holds. Also assume there exist $p,\, q\in(1,\infty]$ satisfying $\frac{1}{p}+\frac{1}{q}<\frac{2}{d}$ such that $M_2(p,q)<\infty$. Then for $x_0\in\bbR^d$, $0<r_1<r_2$ and the annulus $A:=\{ x\in\bbR^d\, :\, 0<r_1\leq d(x_0,x)\leq r_2 \}$,
\begin{equation}
\label{eq:green scaling}
\lim_{n \to \infty} \sup_{x\in A} \lvert n^{d-2}g^\omega(x_0,nx)-a \, g_{\textrm{BM}}(x_0,x)\rvert =0\quad \text{for }\bbP\text{-a.e.~}\omega.
\end{equation}
In the above we have $a:=\bbE[\Lambda^\omega(0)]^{-1}$.
\end{theorem}
\begin{remark}
Analogous results have been proven for the RCM in \cite[Theorem~1.14]{harnack} and \cite[Theorem~5.3]{Ger20}. See also \cite[Theorem~1.6]{AH20} for further estimates on the Green's function which can be derived from off-diagonal heat kernel estimates.
\end{remark}
\noindent\textbf{Notation and Structure of the Paper.} For $x\in\bbR^d$, $\abs{x}$ denotes the standard Euclidean norm. For vectors $u,\,v\in\bbR^d$, the canonical scalar product is given by $u\cdot v$ and gradient $\nabla u$. We write $c$ to denote a positive, finite constant which may change on each appearance. Constants denoted by $c_i$ will remain the same. For $\al,\,\beta\in\bbR$, we write $\al\simeq\beta$ to mean there exist constants $c,\,\tilde{c}>0$ such that $c\,\al\leq \beta\leq \tilde{c}\,\al$. For a countable set $A$, its cardinality is denoted $\abs{A}$. Otherwise if $A\subset \bbR^d$, $\abs{A}$ is the Lebesgue measure. For any $p\in(1,\infty)$, the H\"older conjugate is written $p_*:=\frac{p}{p-1}$. We will work with inner products as follows, for functions $f,\,g:\bbR^d\to\bbR$ and positive weight $\nu:\bbR^d\to\bbR$,
\begin{equation}
\nonumber
(f,g):= \int_{\bbR^d}f(x)g(x)\,dx, \qquad (f,g)_\nu:= \int_{\bbR^d}f(x)g(x)\nu(x)\,dx.
\end{equation}
Furthermore, for $p\in(0,\infty)$ and bounded $B\subset \bbR^d$, define norms
\begin{align*}
\norm{f}_p&:=\Big(\int_{\bbR^d}\abs{f(x)}^p\,dx\Big)^{1/p},\qquad && \hspace{9pt} \norm{f}_{p,\nu}:=\Big(\int_{\bbR^d}\abs{f(x)}^p\nu(x)\,dx\Big)^{1/p},\\
\norm{f}_{p,B}&:=\Big(\frac{1}{\abs{B}}\int_{B}\abs{f(x)}^p\,dx\Big)^{1/p},\qquad &&\norm{f}_{p,B,\nu}:=\Big(\frac{1}{\abs{B}}\int_{B}\abs{f(x)}^p\nu(x)\,dx\Big)^{1/p}.
\end{align*}
For $q\in(0,\infty)$, $I\subset \bbR$, $B\subset\bbR^d$, $Q=I\times B$ and $u:\bbR\times\bbR^d\to\bbR$, let
\begin{align*}
\norm{u}_{p,q,Q}&:=\Big(\frac{1}{\abs{I}}\int_{I}\norm{u_t}_{p,B}^q\,dt\Big)^{1/q},\qquad \norm{u}_{p,q,Q,\nu}:=\Big(\frac{1}{\abs{I}}\int_{I}\norm{u_t}_{p,B,\nu}^q\,dt\Big)^{1/q},\\
\norm{u}_{p,\infty,Q}&:=\esssup_{t\in I}\norm{u_t}_{p,B}, \qquad \hspace{13pt} \norm{u}_{p,\infty,Q,\nu}:=\esssup_{t\in I}\norm{u_t}_{p,B,\nu}.
\end{align*}
Finally for $f:\bbR^{d}\to\bbR^{m}$, let
$$\norm{f}_{\infty}:=\esssup_{x\in\bbR^{d}}\abs{f(x)}.$$

All of the results herein will be quenched, in that they hold for $\bbP$-a.e.~instance of the environment $\om$ unless stated otherwise. Regarding the structure of the paper, Section~\ref{section:davies} is devoted to the proof of the upper off-diagonal heat kernel estimate Theorem~\ref{thm:off-diagonal estimate}. The lower estimate, Theorem~\ref{thm:lower estimate}, is then proven in Section~\ref{section:lower estimate}. Finally, Section~\ref{section:green kernel} concerns the proof of the Green's function scaling limit.

\section{Davies' Method}
\label{section:davies}
Throughout this section assume $d\geq 2$, Assumption~\ref{ass:ergodicity} holds and let $p,\,q,\,r\in(1,\infty]$ satisfy $\frac{1}{r}+\frac{1}{q}+\frac{1}{p-1}\frac{r-1}{r}<\frac{2}{d}$. One important space we will work with is $\cF^{\th}_G$ which, for open $G\subseteq \bbR^d$, is the closure of $C_0^\infty(G)$ in $L^2(G,\th^\om dx)$ with respect to $\cE^\om+(\cdot\, ,\cdot)_\th$. We write $\cF^\th$ in the case $G=\bbR^d$ and if $\th^\om\equiv 1$ also we simply write $\cF$. Define $\cF^\th_{\textrm{loc}}$ by $u\in\cF^\th_{\textrm{loc}}$ if for all balls $B\subset\bbR^d$ there exists $u_B\in \cF_B^\th$ such that $u=u_B$ $\bbP$-a.s. In the case $\th^\om\equiv 1$ this space is denoted $\cF_{\textrm{loc}}$. Following \cite{CD15}, we define the parabolic equation that the heat kernel formally satisfies.

\begin{definition}[Caloric function]
\label{def:caloric}
Let $I\subseteq\bbR$ and $G\subseteq \bbR^d$ be open sets. A function $u:I\to \cF^\th_G$ is caloric if the map $t \mapsto (u(t,\cdot),\phi)_\th$ is differentiable for any $\phi \in L^2(G,\th^\om dx)$ and
\begin{equation}
\label{eq:caloric defn}
\frac{d}{dt}(u_t,\phi)_\th+\cE^\om(u_t,\phi)= 0,
\end{equation}
for all $\phi \in \cF_G^\th$.
\end{definition}

\subsection{Maximal Inequality for the Perturbed Cauchy Equation}
\label{subsection:cauchy problem}

The first step in applying Davies' method is to establish a bound on solutions to the following Cauchy problem.

\begin{lemma}[Cauchy Problem]
\label{lemma:cauchy estimate}
Let $u$ be caloric on $\bbR\times\bbR^d$ and $u(0,\cdot)=f(\cdot)$ for some $f \in L^2(\bbR^d,\th^\om dx)$. Let $\psi\in W^{1,\infty}_{\textrm{loc}}(\bbR^d)$ satisfy $\norm{\psi}_\infty<\infty$ and
$$h^\om(\psi)^2:=\esssup_{x\in\bbR^d}\frac{(\grad\psi\cdot a^\om\grad\psi)(x) }{\th^\om(x)}<\infty.$$
Then writing $v(t,x):=\me^{\psi(x)}u(t,x)$, we have for all $t>0$,
$$\norm{v_t}_{2,\th}^2\leq \me^{h^\om(\psi)^2t}\norm{\me^\psi f}_{2,\th}^2.$$

\end{lemma}
\begin{proof}
Formally, $v_t=v(t,\cdot)$ solves the caloric equation
\begin{equation}
\label{eq:caloric J}
\frac{d}{dt}(v_t,\phi)_\th+J^\om(v_t,\phi)=0,
\end{equation}
where we have defined an operator
$$J^\om(v,\phi):=\int_{\bbR^d}(a^\om\nabla v)\cdot\nabla\phi+\phi(a^\om\nabla v)\cdot\nabla \psi-v(a^\om\nabla \psi)\cdot\nabla \phi-v\phi (a^\om\nabla\psi)\cdot\nabla \psi \,dx.$$
More precisely, for $t>0$, $u_t=u(t,\cdot)\in\cF^\th$ by Definition~\ref{def:caloric} and the supposed properties of $\psi$ guarantee that $\me^{2\psi}u_t\in \cF^\th$ also. Therefore, setting $\phi=\me^{2\psi}u_t$ in \eqref{eq:caloric defn} and rearranging, we have
$$\frac{d}{dt}(v_t,v_t)_\th+J^\om(v_t,v_t)=0.$$
Since $a^\om$ is symmetric,
\begin{align*}
J^\om(v_t,v_t)&=\int_{\bbR^d}(a^\om\nabla v_t)\cdot\nabla v_t-v_t^2(a^\om\nabla\psi)\cdot\nabla\psi\, dx\\
&\geq -\int_{\bbR^d}v_t^2(a^\om\nabla\psi)\cdot\nabla\psi\, dx\\
&\geq -h^\om(\psi)^2\, \norm{v_t}_{2,\th}^2.
\end{align*}
Therefore,
$$\frac{d}{dt}\norm{v_t}_{2,\th}^2\leq h^\om(\psi)^2\norm{v_t}_{2,\th}^2,$$
from which the result follows.
\end{proof}

We now establish an energy estimate which we will go on to apply iteratively in order to derive a maximal inequality for $v$.

\begin{lemma}
\label{lemma:energy estimate}
Let $I=(t_1,t_2)\subseteq \bbR_+$ and $B\subseteq \bbR^d$ be any Euclidean ball, $Q:=I\times B$. Let $u$ be a locally bounded positive caloric function on $I\times B$, and $v(t,x):=\me^{\psi(x)}u(t,x)$ for $\psi\in W_{\textrm{loc}}^{1,\infty}(\bbR^d)$ with $\norm{\psi}_\infty<\infty$ and $h^\om(\psi)^2<\infty$. Define cut-off functions $\eta \in C_0^\infty (B)$ such that $0\leq \eta\leq 1$ and $\xi:\bbR\to [0,1]$ with $\xi\equiv 0$ on $(-\infty, t_1]$. Then, there exists $c_5>1$ such that for any $\alpha \geq 1$,
\begin{multline}
\nonumber
\frac{1}{\abs{I}} \norm{\xi(\eta v^\al)^2}_{1,\infty, Q, \theta}+\frac{1}{\abs{I}}\int_I \xi(t)\frac{\cE^\om(\eta v_t^\al,\eta v_t^\al)}{\abs{B}}\, dt\\
\leq c_5\,\Big(\norm{\nabla \eta}_{\infty}^2 \norm{\La^\om/\th^\om}_{p,B,\th}\norm{v^{2\al}}_{p_*,1,Q,\th}+ \big(\al\, h^\om(\psi)^2+\norm{\xi'}_{\infty}\big)\norm{(\eta v^\al)^2}_{1,1,Q, \th}\Big).
\end{multline}
\end{lemma}

\begin{proof}
One can show using \eqref{eq:caloric J} and the same argument as \cite[Lemma~B.3]{CD15} that
\begin{equation}
\label{eq:energy1}
\frac{d}{dt}(v_t^{2\al},\eta^2)_\th +2\al\, J^\om(v_t, \eta^2 v_t^{2\al -1})\leq 0\quad \forall\,t\geq 0.
\end{equation}
Then
\begin{align*}
\al\, J^\om(v_t,\eta^2v_t^{2\al-1})&=\int_{\bbR^d}\al\,(a^\om\nabla v_t)\cdot\nabla(\eta^2 v_t^{2\al-1}) + \al\,\eta^2 v_t^{2\al-1} (a^\om\nabla v_t)\cdot \grad\psi \\
& \qquad -\al\, v_t(a^\om\grad\psi)\cdot\grad(\eta^2v_t^{2\al-1})-\al\, \eta^2v_t^{2\al}(a^\om\grad\psi)\cdot\grad\psi\, dx.
\end{align*}
We label these integrands $J_1,\dots, J_4$ in order.
\begin{align*}
J_1&=\al\,(a^\om\grad v_t)\cdot \grad(\eta^2v_t^{2\al-1})\\
&=\eta^2(a^\om\grad v_t)\cdot\grad(\al v_t^{2\al-1})+2\al\,\eta\, v_t^{2\al-1}(a^\om\grad v_t)\cdot \grad \eta.
\end{align*}
By algebraic manipulation,
$$J_1=\frac{2\al-1}{\al}(a^\om\grad (\eta v_t^\al))\cdot \grad (\eta v_t^\al)-\frac{\al-1}{\al}v_t^\al(a^\om\grad(\eta v_t^\al))\cdot \grad \eta-\frac{1}{\al}(a^\om\grad\eta)\cdot\grad(\eta v_t^{2\al}).$$
Then since $\al\geq 1$,
\begin{align}
\nonumber
J_1&\geq (a^\om\grad(\eta v_t^\al))\cdot \grad (\eta v_t^\al)-v_t^\al\abs{(a^\om\grad(\eta v_t^\al))\cdot\grad\eta}-v_t^{2\al}(a^\om\grad\eta) \cdot\grad \eta\\
&\geq \frac{1}{2}(a^\om\grad(\eta v_t^\al))\cdot \grad (\eta v_t^\al)-2\, v_t^{2\al}(a^\om\grad\eta) \cdot \grad\eta.
\end{align}
Similarly,
\begin{align}
\abs{J_2}&\leq \frac{1}{8}(a^\om\grad(\eta v_t^\al))\cdot\grad(\eta v_t^\al)+3 \eta  v_t^{2\al}(a^\om\grad\psi)\cdot \grad\psi+v_t^{2\al} (a^\om\grad\eta)\cdot\grad\eta.\\
\nonumber
\abs{J_3}&\leq \frac{1}{8}(a^\om\grad(\eta v_t^\al))\cdot\grad(\eta v_t^\al)+8\al\eta^2 v^{2\al}(a^\om\grad\psi)\cdot\grad\psi\\
&\qquad \qquad \qquad \qquad + \eta^2 v_t^{2\al}(a^\om\grad\psi) \cdot \grad\psi+v_t^{2\al}(a^\om\grad\eta)\cdot\grad\eta.
\end{align}
Substituting the above estimates into \eqref{eq:energy1},
\begin{multline}
\frac{d}{dt}\norm{\eta^2 v_t^{2\al}}_{1,\th} \leq \int_{\bbR^d}4\,v_t^{2\al}(a^\om\grad\eta)\cdot\grad\eta-\frac{1}{4}(a^\om\grad(\eta\, v_t^\al))\cdot\grad(\eta\, v_t^\al)\, dx\\
+(9\al+4)\int_{\bbR^d}\eta^2 v_t^{2\al}(a^\om\grad\psi ) \cdot \grad\psi\, dx.
\end{multline}
Therefore,
\begin{align*}
\frac{d}{dt}\norm{\eta^2v_t^{2\al}}_{1,B,\th}+&\frac{1}{4}\, \cE^\om(\eta v_t^\al,\eta v_t^\al)\\
&\leq 4\int_{\bbR^d}v_t^{2\al}(a^\om\grad\eta)\cdot\grad\eta\,dx
+(9\al+4)\int_{\bbR^d}\eta^2 v_t^{2\al}(a^\om\grad\psi) \cdot \grad\psi\, dx.
\end{align*}
We can then bound these terms as follows
\begin{align*}
\int_{\bbR^d}v_t^{2\al}(a^\om\grad\eta)\cdot\grad\eta\,dx &\leq\norm{\grad\eta}_\infty^2 \int_{B} v_t^{2\al} \La^\om\, dx\leq \norm{\grad\eta}_\infty^2\norm{v_t^{2\al}}_{1,B,\La}\\
&\leq \norm{\grad\eta}_{\infty}^2\norm{\La^\om/\th^\om}_{p,B,\th}\norm{v_t^{2\al}}_{p_*,B,\th}.\\
\int_{\bbR^d}\eta^2 v_t^{2\al}(a^\om\grad\psi) \cdot \grad\psi\, dx &\leq h^\om(\psi)^2 \int_{\bbR^d}\eta^2 v_t^{2\al}\th^\om \, dx\\
&= h^\om(\psi)^2\norm{\eta^2v_t^{2\al}}_{1,B,\th},
\end{align*}
where we used H\"older's inequality on the first term. So,
\begin{multline}
\frac{d}{dt}\norm{\eta^2 v_t^{2\al}}_{1,B,\th}+\frac{1}{4}\cE^\om(\eta v_t^\al,\eta v_t^\al)\\
\leq 4\,\norm{\grad\eta}_{\infty}^2\norm{\La^\om/\th^\om}_{p,B,\th}\norm{v_t^{2\al}}_{p_*,B,\th}+(9\al+4)h^\om(\psi)^2\norm{\eta^2v_t^{2\al}}_{1,B,\th}.
\end{multline}
Now let $t\in (t_1,t_2)$, multiply the above by $\xi(s)$ and integrate from $s=t_1$ to $s=t$,
\begin{multline}
\frac{1}{\abs{I}}\Big(\xi(t)\norm{\eta^2 v_t^{2\al}}_{1,B,\th}+\frac{1}{4}\int_{t_1}^t\cE^\om(\eta v_s^\al,\eta v_s^\al)\, ds\Big)\\
\leq 4\norm{\grad\eta}_{\infty}^2\norm{\La^\om/\th^\om}_{p,B,\th}\norm{v^{2\al}}_{p_*,1,I\times B,\th}+(9\al+4)h^\om(\psi)^2\norm{\eta^2v^{2\al}}_{1,1,I\times B,\th}\\
+\sup_{s\in I}\abs{\xi'(s)}\norm{\eta^2v^{2\al}}_{1,1,I\times B,\th}.
\end{multline}
Note that the final term on the right-hand side appears by integration by parts with the first term on the left-hand side. Finally, take supremum over $t\in I$ on the left-hand side.
\end{proof}

The following Sobolev inequality is another component in deriving the maximal inequality in Proposition~\ref{prop:maximal inequality}.

\begin{lemma}[Sobolev Inequality]
\label{lemma:sobolev}
Let $B\subseteq \bbR^d$ be a Euclidean ball and $\eta\in C_0^\infty(B)$ a cut-off function. Then there exists $c_6(d,q)>0$ such that for all $u \in \cF_{\textrm{loc}}^\th \cup \cF_{\textrm{loc}}$,
\begin{equation}
\label{eq:sobolev}
\norm{\eta^2 u^2}_{\rho/r_*,B, \th}\leq c_6 \, \abs{B}^{\frac{2}{d}}\norm{(\la^\om)^{-1}}_{q,B}\norm{\th^\om}_{r,B}^{r_*/\rho}\frac{\cE^\om(\eta u, \eta u)}{\abs{B}},
\end{equation}
where $\rho := qd/(q(d-2)+d)$.
\end{lemma}
\begin{proof}
Firstly, by H\"older's inequality,
\begin{equation}
\label{eq:sobolev1}
\norm{\eta^2 u^2}_{\rho/r_*,B, \th}\leq \norm{\th^\om}_{r,B}^{r_*/\rho}\norm{\eta^2 u^2}_{\rho,B}.
\end{equation}
Also by \cite[Proposition 2.3]{CD15},
$$\norm{\eta^2 u^2}_\rho \leq \norm{\indic_B (\la^\om)^{-1}}_{q}\cE^\om(\eta u, \eta u).$$
%
After averaging over $B$ this yields
\begin{equation}
\label{eq:sobolev2}
\norm{\eta^2 u^2}_{\rho,B}\leq c\, \abs{B}^{2/d} \norm{(\la^\om)^{-1}}_{q,B}\frac{\cE^\om(\eta u, \eta u)}{\abs{B}},
\end{equation}
for some $c=c(d,q)>0$. The result then follows from \eqref{eq:sobolev1} and \eqref{eq:sobolev2}.
\end{proof}

We now derive the maximal inequality for $v$ using Moser iteration. For $x_0 \in \bbR^d,\, \de \in (0,1]$ and $n\in\bbR_+$ we denote a space-time cylinder $Q_\de(n):=[0,\de n^2]\times B(x_0,n)$. Furthermore, for $\si \in(0,1]$ and $\eps \in(0,1]$, let $s'=\eps\de n^2$, $s''=(1-\eps)\de n^2$ and define
\begin{equation}
\label{eq:cylinder defn}
Q_{\de,\si}(n):=[(1-\si)s',(1-\si)s''+\si \de n^2]\times B(x_0,\si n).
\end{equation}

\begin{proposition}
\label{prop:maximal inequality}
Let $x_0\in \bbR^d,\,\de\in(0,1],\,\eps\in(0,1/4),\, 1/2\leq \si'< \si \leq 1$ and $n\in[1,\infty)$. Let $v$ be as in Lemma \ref{lemma:energy estimate}. Then there exist constants $c_7(d,p,q,r)$ and $\kappa(d,p,q,r)$ such that
\begin{equation}
\label{eq:maximal inequality}
\max_{(t,x)\in Q_{\de, 1/2}(n)}v(t,x)\leq c_7\,\Big((1+\de n^2 h^\om(\psi)^2)\frac{\cA^\om(n)}{\eps(\si-\si')^2}\Big)^{\frac{\kappa}{p_*}}\norm{v}_{2p_*,2,Q_{\de,\si}(n),\th}.
\end{equation}
In the above, $\cA^\om(n):=\norm{1\vee (\La^\om/\th^\om)}_{p,B(x_0,n),\th}\norm{1\vee (\la^\om)^{-1}}_{q,B(x_0,n)}\norm{1\vee \th^\om}_{r,B(x_0,n)}$.
\end{proposition}
\begin{proof}
Define $\alpha:= 1+\frac{1}{p_*}-\frac{r_*}{\rho}>1$
and write $\alpha_k:=\alpha^k$ for $k\in\bbN$. Let $\si_k:= \si'+2^{-k}(\si-\si')$ and $\tau_k:=2^{-k-1}(\si-\si')$. Also introduce shorthand $I_k=[(1-\si_k)s', (1-\si_k)s''+\si_k\de n^2]$, $B_k:=B(x_0,\si_k n)$ and $Q_k=I_k\times B_k= Q_{\de, \si_k}(n)$. Note that $\abs{I_k}/\abs{I_{k+1}}\leq 2$ and $\abs{B_k}/\abs{B_{k+1}}\leq c\, 2^d $. 
We begin by applying H\"older's and Young's inequalities,
\begin{equation}
\label{eq:moser1}
\norm{v^{2\al_k}}_{\al p_*,\al, Q_{k+1},\th}\leq \norm{v^{2\al_k}}_{1,\infty,Q_{k+1},\th}+\norm{v^{2\al_k}}_{\rho/r_*, 1, Q_{k+1},\th},
\end{equation}
with $\rho$ as in Lemma \ref{lemma:sobolev}. Now let $k\in\bbN$ and define a sequence of cut-off functions in space, $\eta_k:\bbR^d \to [0,1]$ such that $\supp \eta_k\subseteq B_k$, $\eta_k\equiv 1 $ on $B_{k+1}$ and $\norm{\grad \eta_k}_{\infty}\leq \frac{2}{\tau_k n}$. Similarly, let $\xi_k:\bbR\to [0,1]$ be time cut-offs such that $\xi_k\equiv 1$ on $I_{k+1}$, $\xi_k\equiv 0 $ on $(-\infty, (1-\si_k)s']$ and $\norm{\xi'}_\infty \leq \frac{2}{\tau_k\de n^2}$. Then by \eqref{eq:moser1},
\begin{equation}
\label{eq:moser2}
\norm{v^{2\al_k}}_{\al p_*,\al, Q_{k+1},\th}\leq c\,\Big(\norm{\xi_k(\eta_kv^{\al_k})^2}_{1,\infty,Q_{k},\th}+\norm{\xi_k(\eta_kv^{\al_k})^2}_{\rho/r_*, 1, Q_{k},\th} \Big) .
\end{equation}
We will bound both terms on the right-hand side. By the Sobolev inequality \eqref{eq:sobolev},
\begin{multline}
\norm{\xi_k(\eta_kv^{\al_k})^2}_{\rho/r_*, 1, Q_{k},\th}\\
\leq c\, n^2\, \norm{(\la^\om)^{-1}}_{q,B_k}\norm{\th^\om}_{r,B_k}^{r_*/\rho}\, \frac{1}{\abs{I_k}}\int_{I_k}\xi_k(t)\frac{\cE^\om(\eta_k v_t^{\al_k}, \eta_k v_t^{\al_k})}{\abs{B_k}}\, dt.
\end{multline}
So,
\begin{multline}
\label{eq:moser3}
\norm{\xi_k(\eta_kv^{\al_k})^2}_{1,\infty,Q_{k},\th}+\norm{\xi_k(\eta_kv^{\al_k})^2}_{\rho/r_*, 1, Q_{k},\th}\leq \frac{c\, n^2}{\abs{I_k}} \norm{\xi_k(\eta_kv^{\al_k})^2}_{1,\infty,Q_{k},\th}\\
  +  \frac{c\, n^2}{\abs{I_k}}\norm{(\la^\om)^{-1}}_{q,B_k}\norm{\th^\om}_{r,B_k}^{r_*/\rho}\int_{I_k}\xi_k(t)\frac{\cE^\om(\eta_k v_t^{\al_k}, \eta_k v_t^{\al_k})}{\abs{B_k}}\, dt.
\end{multline}
By Lemma~\ref{lemma:energy estimate} and H\"older's inequality,
\begin{align*}
\eqref{eq:moser3}&\leq c\,\alpha_k\, \cA^\om(n)\Big(\frac{1}{\de \tau_k^2}+n^2 h^\om(\psi)^2\Big)\norm{v^{2\al_k}}_{p_*, 1, Q_{k},\th}.
\end{align*}
Returning to \eqref{eq:moser1},
\begin{align*}
\norm{v}_{2\al_{k+1} p_*,2\al_{k+1}, Q_{k+1},\th}
&=\norm{v^{2\al_k}}_{\al p_*,\al, Q_{k+1},\th}^{1/(2\al_k)}\\
&\leq \Big(c\,\al_k\, 2^{2k}\frac{(1+\de n^2 h^\om(\psi)^2)}{\de (\si-\si')^2}\cA^\om(n)\Big)^{\frac{1}{2\al_k}}\norm{v}_{2\al_{k} p_*,2\al_{k}, Q_{k},\th}.
\end{align*}
Iterating the above, for any $K\in\bbZ_+$,
\begin{multline}
\nonumber
\norm{v}_{2\al_{K} p_*,2\al_{K}, Q_{K},\th} \leq c\, \prod_{k=0}^{K-1}\Big(\al_k\, 2^{2k}\frac{(1+\de n^2 h^\om(\psi)^2)}{\de (\si-\si')^2}\cA^\om(n)\Big)^{\frac{1}{2\al_k}}\, \norm{v}_{2p_*,2, Q_{\de,\si}(n),\th}.
\end{multline}
Sending $K\to\infty$, observing that $Q_K\downarrow Q_{\de,\frac{1}{2}}(n)$ and $\prod_{k=0}^{K-1} (\al_k 2^{2k})^{\frac{1}{2\al_k}}$ is uniformly bounded in $K$, we have
\begin{equation}
\max_{(t,x)\in Q_{\de,\frac{1}{2}}(n)} v(t,x)\leq c\, \Big((1+\de n^2 h^\om(\psi)^2)\frac{\cA^\om(n)}{\eps(\si-\si')^2}\Big)^{\frac{\ka}{p_*}}\norm{v}_{2p_*,2, Q_{\de,\si}(n),\th},
\end{equation}
where $\kappa:=\frac{p_*}{2}\sum_{k=0}^\infty \frac{1}{\al_k}<\infty$.
\end{proof}

\begin{corollary}
\label{cor:maximal l2}
In the same setting as Proposition~\ref{prop:maximal inequality}, there exists $c_8(d,p,q,r)>0$ such that
\begin{equation}
\label{eq:maximal l2}
\max_{(t,x)\in Q_{\de, \frac{1}{2}}(n)}v(t,x)\leq c_8\,\Big((1+\de n^2 h^\om(\psi)^2)\frac{\cA^\om(n)}{\eps(\si-\si')^2}\Big)^{\kappa}\norm{v}_{2,\infty,Q_{\de}(n),\th}.
\end{equation}
\end{corollary}
\begin{proof}
This is derived from Proposition~\ref{prop:maximal inequality}, in a similar fashion to \cite[Theorem 2.2.3]{DS96}.
\end{proof}

\subsection{Heat Kernel Bound}
We first conglomerate the two results of the preceding section -- the Cauchy problem estimate and the maximal inequality.
\begin{proposition}
\label{prop:maximal cauchy}
In the same setting as Proposition~\ref{prop:maximal inequality}, there exists $c_9(d,p,q,r,\eps)>0$ such that
\begin{equation}
\label{eq:maximal cauchy}
\max_{(t,x)\in Q_{\de,\frac{1}{2}}(n)}v(t,x)\leq \frac{c_9}{n^{d/2}}\Big(\frac{\cA^\om(n)}{\eps\de}\Big)^\kappa \me^{2(1-\eps)h^\om(\psi)^2\de n^2}\norm{\me^\psi f}_{2,\th}.
\end{equation}
\end{proposition}
\begin{proof}
By combining Corollary~\ref{cor:maximal l2} with Lemma~\ref{lemma:cauchy estimate}, we obtain
$$\max_{(t,x)\in Q_{\de,\frac{1}{2}}(n)}v(t,x)\leq \frac{c}{n^{d/2}}\Big((1+\de n^2h^\om(\psi)^2)\frac{\cA^\om(n)}{\eps\de}\Big)^\kappa \me^{h^\om(\psi)^2\de n^2/2}\norm{\me^\psi f}_{2,\th}.$$
The result follows since for any $\eps\in(0,1/2)$ there exists $c(\eps)<\infty$ such that
$$(1+\de n^2h^\om(\psi)^2)^\kappa\leq c(\eps)\, \me^{(1-2\eps)h^\om(\psi)^2\de n^2},$$
for all $n\geq 1,\,\de\in(0,1]$.
\end{proof}

\begin{proposition}[Heat Kernel Bound]
\label{prop:davies hke}
Suppose $M_1(p,q,r)<\infty$ and let $x_0\in \bbR^d$. Then $\bbP$-a.s.~there exist $c_{10}(d,p,q,r),\,\gamma(d,p,q,r)>0$ such that for all $\sqrt{t}\geq N_2^\om(x_0)$ and $x,\,y\in \bbR^d$,
\begin{equation}
\label{eq:heat kernel1}
p_\th^\om(t,x,y)\leq c_{10}\,  t^{-\frac{d}{2}} \Big(1+\frac{d(x_0,x)}{\sqrt{t}}\Big)^\gamma\Big(1+\frac{d(x_0,y)}{\sqrt{t}}\Big)^\gamma \me^{2h^\om(\psi)^2 t-\psi(x)+\psi(y)}.
\end{equation}
\end{proposition}

\begin{proof}
Fix $\eps=\frac{1}{8}$. By the ergodic theorem there exists $N_2^\om(x_0)>0$ such that 
\begin{equation}
\label{eq:random radius}
\nonumber
\cA^\om(n)\leq c\, \big(1+\bbE[\La^\om(0)^p\th^\om(0)^{1-p}]\big)\big(1+\bbE[\la^\om(0)^{-q}]\big)\big(1+\bbE[\th^\om(0)^r]\big)=:\bar{A}<\infty,
\end{equation}
for all $n\geq N_2^\om(x_0)$. For given $x\in\bbR^d$ and $\sqrt{t}> N_2^\om(x_0)$, we choose $\de, \, n$ such that $(t,x)\in Q_{\de,\frac{1}{2}}(n)$, for example by setting $n=2d(x_0,x)+\sqrt{8t/7}$ and $\de:=8t/(7n^2)$. Then considering the caloric function $u(t,x):=P_tf(x)$ for $f \in \cF^\th$, by Proposition~\ref{prop:maximal cauchy},
\begin{align}
\nonumber
\me^{\psi(x)}u(t,x)& \leq c\, n^{-d/2}(n^2/t)^\kappa \me^{2h^\om(\psi)^2 t}\norm{\me^\psi f}_{2,\th}\\
&\leq c\, n^{\gamma}t^{-\kappa}\me^{2h^\om(\psi)^2t}\norm{\me^\psi f}_{2,\th},
\end{align}
for some $c=c(\eps,d,p,q,r,\bar{A})$, where $\gamma:=2\kappa-\frac{d}{2}$. Write $r(t):=c\, t^{-\kappa}\me^{2h^\om(\psi)^2 t}$ and $b_t(x):=\big(2d(x_0,x)+\sqrt{8t/7}\big)^\gamma.$ Since the above holds for all $x\in\bbR^d$ and $\sqrt{t}>N_2^\om(x_0)$, we have
$$\me^{\psi(x)}P_t f(x)\leq b_t(x)r(t)\norm{\me^\psi f}_{2,\th}.$$
That is,
\begin{equation}
\label{eq:operator1}
\norm{b_t^{-1}\me^{\psi}P_t f}_{\infty}\leq r(t)\norm{\me^{\psi}f}_{2,\th}.
\end{equation}
Now define an operator $P_t^\psi(g):=\me^\psi P_t(\me^{-\psi}g)$ for $\me^{-\psi}g\in\cF^\th$ . Then we can bound the operator norm
$$\norm{b^{-1}_t P_t^\psi(\me^\psi f)}_{L^2(\bbR^d, \th^\om  dx)\to L^\infty}\leq r(t).$$
The above also holds with $\psi$ replaced by $-\psi$. Since the dual of $P_t^\psi$ is $P_t^{-\psi}$, the dual of $b^{-1}_tP_t^{-\psi}(\cdot)$ is $P_t^{\psi}(b^{-1}_t\cdot)$. So by duality,
\begin{equation}
\label{eq:operator2}
\norm{P_t^\psi(b^{-1}_tg)}_{2,\th}\leq r(t)\norm{g}_{1,\th}.
\end{equation}
Since $b_{\frac{t}{2}}(x)\leq b_t(x)$, we have
\begin{align}
\nonumber
\norm{b_t^{-1}\me^\psi P_t f}_{\infty}&\leq\norm{b_{\frac{t}{2}}^{-1}\me^\psi P_{\frac{t}{2}}P_{\frac{t}{2}} f}_{\infty}\\
\nonumber
&\leq r(t/2)\,\norm{\me^\psi P_{\frac{t}{2}} f}_{2,\th}\quad \text{ by \eqref{eq:operator1},}\\
&\leq r(t/2)^2\,\norm{\me^\psi b_\frac{t}{2} f}_{1,\th}\quad \text{ by \eqref{eq:operator2}.}
\end{align}
That is, for all $x\in\bbR^d$ and $\sqrt{t}\geq N_2^\om(x_0)$, we have
\begin{equation}
\label{eq:semigroup estimate}
\nonumber
P_t f(x)\leq \frac{c}{t^{2\kappa}}\me^{2h^\om(\psi)^2 t-\psi(x)}(d(x_0,x)+\sqrt{t})^\gamma \int_{\bbR^d}(d(x_0,y)+\sqrt{t})^\gamma \me^{\psi(y)} \abs{f(y)}\th^\om(y)\, dy.
\end{equation}
It is standard that the above implies the heat kernel estimate \eqref{eq:heat kernel1} for almost all $x,\,y\in\bbR^d$. Furthermore, local boundedness in Assumption~\ref{ass:ergodicity} allows us to pass to all $x,\,y\in\bbR^d$.
\end{proof}

\subsection{Properties of the Intrinsic Metric}
\label{subsection:intrinsic}
In order to prove the off-diagonal estimate in Theorem~\ref{thm:off-diagonal estimate} from Proposition~\ref{prop:davies hke}, we aim to set the function $\psi(\cdot)=\beta\,d_\th^\om(x,\cdot)$ in \eqref{eq:heat kernel1}, then optimise over the constant $\beta$. This requires checking that this function $\psi$ satisfies the necessary regularity assumptions for the proofs in Section~\ref{subsection:cauchy problem}. Recall that the intrinsic metric is defined as follows,
$$
d^\om_\th(x,y):=\sup\Big\lbrace \phi(y)-\phi(x) \, :\, \phi \in C^1(\bbR^d),\, h^\om(\phi)^2= \sup_{z\in\bbR^d}\frac{(\grad\phi\cdot a^\om\grad\phi)(z)}{\th^\om(z)}\leq 1\Big\rbrace.
$$

In deriving the required regularity of $d_\th^\om$, we first show that it is equal to $D_\th^\om$, the Riemannian distance computed with respect to $(\frac{a^\om}{\th^\om})^{-1}$. This Riemannian metric is defined via the following path relation. Consider the following Hilbert space
$$\cH:=\big\{ f\in C\big([0,\infty),\bbR^d\big)\, : \, f(0)=0,\, \dot{f}\in L^2\big([0,\infty), \bbR^d\big)\big\},$$
where $\dot{f}$ denotes the weak derivative of $f$, together with the following norm
$$\norm{f}_{\cH}:=\norm{\dot{f}}_{L^2([0,\infty),\bbR^d)}.$$
Given $f\in \cH$, define $\Phi(t,x;f):[0,\infty)\times \bbR^d \to \bbR^d$ via
$$\frac{d}{dt}\Phi(t,x;f)=\Big(\frac{a^\om(\Phi(t,x;f))}{\th^\om(\Phi(t,x;f))}\Big)^{1/2}\dot{f}(t),$$
with initial condition $\Phi(0,x;f)=x$. The Riemannian distance is then given by
$$D_\th^\om(x,y):=t^{1/2}\inf \Big\lbrace \norm{f}_\cH\, :\, f\in \cH,\, \Phi(t,x;f)=y\Big\rbrace,$$
for any $t>0$.

\begin{lemma}[Riemannian Distance Representation]
\label{lemma:riemannian distance}
For all $x,\, y\in\bbR^d$, $d_\th^\om(x,y)=D_\th^\om(x,y)$.
\end{lemma}
\begin{proof}
This follows by the proof of \cite[Lemma I.1.24]{Str88}.
\end{proof}
Next we will apply the additional Assumption~\ref{ass:continuity} on the environment to derive the regularity we require of $d_\th^\om$. Our objective is to pass a function resembling $\rho_x(\cdot):=d_\th^\om(x,\cdot)$ into \eqref{eq:heat kernel1}. In order to do this we must show some conditions such as $\rho_x\in W_{\textrm{loc}}^{1,\infty}(\bbR^d)$ and $h^\om(\rho_x)^2\leq 1$. The requisite property is that the metric $d_\th^\om$ is strictly local i.e.~that $d_\th^\om$ induces the original topology on $\bbR^d$. For further discussion of the properties of such intrinsic metrics and the distance function $\rho_x$ see \cite{Sto10}, \cite[Appendix A]{MLS09} and \cite{Stu95}. In the following proposition, we invoke a recent result from geometric analysis to directly deduce strict locality of the intrinsic metric $d_\th^\om$ under Assumption \ref{ass:continuity}.

\begin{proposition}
\label{prop:strict locality}
If Assumption~\ref{ass:continuity} holds then the intrinsic metric $d_\th^\om$ is strictly local for $\bbP$-a.e.~$\om$.
\end{proposition}
\begin{proof}
Given Assumption~\ref{ass:continuity} and Lemma~\ref{lemma:riemannian distance}, this follows directly from Proposition~4.1ii) or Theorem~4.5 in \cite{Bur15}, noting that the Euclidean metric corresponds to the Riemannian metric given by the identity matrix \cite[Proposition~3.3]{Sto10}.
\end{proof}

\subsection{Upper Off-Diagonal Estimate}
\label{subsection:proof of UE}
Having proven the necessary regularity of the intrinsic metric in the preceding subsection, we are now in a position to optimise over the test function in Proposition~\ref{prop:davies hke} and derive the upper off-diagonal estimate.

\begin{proof}[Proof of Theorem~\ref{thm:off-diagonal estimate}]
As a corollary to Proposition~\ref{prop:strict locality}, we have for example by \cite[Lemma 1]{Stu95} that for any $x\in\bbR^d$, $\rho_x \in C(\bbR^d)\cap L^2_{\textrm{loc}}(\bbR^d,\th)$ and $h^\om(\rho_x)^2\leq 1$ almost surely. Furthermore $\rho_x$ has a weak derivative and \cite[Theorem 5.1]{Sto10} implies that $\esssup_{z\in\bbR^d} \abs{\grad\rho_x(z)}<\infty$. The final property to check is that our test function is essentially bounded, whilst $\rho_x$ may be unbounded we can take a bounded version with the desired properties. In accordance with \cite[Eqn.~(2)]{MLS09}, consider $\eta_x=\xi\circ\rho_x$ for a continuously differentiable cut-off function $\xi$ to construct a function such that $\eta_x(x)=d_\th^\om(x,x)=0,\, \eta_x(y)=d_\th^\om(x,y)$, $\eta_x$ is essentially bounded and $\eta_x$ satisfies the aforementioned properties, including $h^\om(\eta_x)^2\leq 1$. This is another consequence of Proposition~\ref{prop:strict locality}.
Therefore we are justified in setting $\psi(\cdot)=-\beta\, \eta_x(\cdot)$ in \eqref{eq:heat kernel1} for $\beta\in\bbR$, and $h^\om(\psi)^2\leq \beta^2$. Then by choosing the constant $\beta=d_\th^\om(x,y)/(4t)$ and setting $x_0=x$ in \eqref{eq:heat kernel1} we have for $\bbP$-a.e.~$\om$, all $x,\, y\in\bbR^d$ and $\sqrt{t}\geq N_2^\om(x)$,
\begin{equation}
\label{eq:heat kernel 2}
p_\th^\om(t,x,y)\leq c\, t^{-\frac{d}{2}}\Big(1+\frac{d(x,y)}{\sqrt{t}}\Big)^\gamma \exp\Big(-\frac{d_\th^\om(x,y)^2}{8t}\Big),
\end{equation}
which completes the proof.
\end{proof}

\section{Lower Off-Diagonal Estimate}
\label{section:lower estimate}
The starting point for proving the lower off-diagonal estimate of Theorem~\ref{thm:lower estimate} is the following near-diagonal estimate. Throughout this section suppose Assumptions~\ref{ass:ergodicity} and \ref{ass:finite range} hold. Also let $p,\,q\in(1,\infty)$ satisfy $\frac{1}{p}+\frac{1}{q}<\frac{2}{d}$.

\begin{proposition}
\label{prop:ndle}
Let $t>0$ and $x\in\bbR^d$, then for all $y\in B\big(x,\frac{\sqrt{t}}{2}\big)$ we have 
\begin{equation}
\label{eq:ndle}
p_{\La}^\om(t,x,y)\geq \frac{t^{-d/2}}{C_{\textrm{PH}}\big(\norm{\Lambda^\om}_{p,B(x,\sqrt{t})},\norm{\lambda^\om}_{q,B(x,\sqrt{t})}\big)}.
\end{equation}
The constant $C_{\textrm{PH}}$ is given explicitly by
\begin{equation}
\label{eq:harnack constant}
C_{\textrm{PH}}=c_{11}\exp\Big(c_{12}\Big(\big(1\vee\norm{\Lambda^\om}_{p,B(x,\sqrt{t})}\big)\big(1\vee\norm{\lambda^\om}_{q,B(x,\sqrt{t})}\big)\Big)^\kappa\Big),
\end{equation}
for $c_i(d,p,q),\, \kappa(d,p,q)>0$.
\end{proposition}
\begin{proof}
A parabolic Harnack inequality with constant $C_{\textrm{PH}}$ is established in \cite[Theorem 3.9]{CD15} and this is a standard consequence of it, see for instance \cite[Proposition 4.7]{harnack} or \cite[Proposition 3.1]{Del99}.
\end{proof}

The chaining method is to apply Proposition~\ref{prop:ndle} along a sequence of balls. Let $x\in\bbR^d$, a radius $0<r\leq 4\, d(0,x)$ and $k\in\bbN$ satisfying $\frac{12d(0,x)}{r}\leq k\leq \frac{16d(0,x)}{r}$. Consider the sequence of points $x_j=\frac{j}{k}x$ for $j=0,\dots, k$ that interpolates between $0$ and $x$. Let $B_{x_j}=B(x_j, \frac{r}{48})$ and $s:=\frac{r\,d(0,x)}{k}$, noting $\frac{r^2}{16}\leq s\leq \frac{r^2}{12}$.

To apply estimate \eqref{eq:ndle} along a sequence we will need to control the ergodic average terms in \eqref{eq:harnack constant} simultaneously for balls with different centre-points. To this end we establish a moment bound in Proposition~\ref{prop:moment bound} which employs finite range dependence to get better control than in the general ergodic setting. First, a prerequisite lemma.

\begin{lemma}
\label{lemma:indep rvs}
For any $k>2$ and independent random variables $Y_1,...,Y_n\in L^k(\bbP)$ with $\bbE[Y_i]=0$ for all $i$, there exists $c_{13}(k)>0$ such that
\begin{equation}
\bbE\Big[\big|\sum_{i=1}^n Y_i\big|^{k}\Big]\leq c_{13}\,\max\Big\{\sum_{i=1}^n\bbE\big[\big|Y_i\big|^{k}\big],\, \Big(\sum_{i=1}^n\bbE\big[\abs{Y_i}^{2}\big]\Big)^{\frac{k}{2}}\Big\}.
\end{equation}
\end{lemma}
\begin{proof}
This follows from \cite[Theorem~3]{Ros70}.
\end{proof}
For $u\in\bbR^d,\, p,\,q>0$ we write $\De\La_p^\om(u):=\La^\om(u)^p-\bbE[\La^\om(0)^p]$ and $\De\la_q^\om(u):=\la^\om(u)^q-\bbE[\la^\om(0)^q]$ for the deviation of these moments from their respective means.

\begin{proposition}
\label{prop:moment bound}
Let $\xi>1$ and assume $M_2(2\xi p,2\xi q)<\infty$. Let $R\subset\bbR^d$ be any region which can be covered by a disjoint partition of $K$ balls of radius $\cR$ in the maximum norm, i.e.~$R\subset\bigcup_{i=1}^K \{z_i+[0,\cR]^d\}$ for some $z_1,\dots,z_K\in\bbR^d$. There exists $c_{14}(d,\cR,\xi)>0$ such that
\begin{align}
\label{eq:La moment bound}
\bbE\Big[\big|\int_R \La^\om(u)^p-\bbE[\La^\om(0)^p]\, du\big|^{2\xi}\Big]&\leq c_{14}\, K^\xi,\\
\label{eq:la moment bound}
\bbE\Big[\big|\int_R \la^\om(u)^q-\bbE[\la^\om(0)^q]\, du\big|^{2\xi}\Big]&\leq c_{14}\, K^\xi.
\end{align}
\end{proposition}
\begin{proof}
We prove the statement only for $\La^\om$, since the one for $\la^\om$ is analogous. Denote $f(u):= \De\La_p^\om(u)\indic_{u\in R}$. Then by Jensen's inequality and Fubini's theorem,
\begin{align}
\nonumber
\bbE\Big[\big|\int_R\De\La_p^\om(u)\, du\big|^{2\xi}\Big]&= \bbE\Big[\big|\int_{[0,\cR]^d}\sum_{i=1}^K f(z_i+u)\, du\big|^{2\xi}\Big]\\
\nonumber
&\leq \cR^{d(2\xi-1)}\bbE\Big[\int_{[0,\cR]^d}\big|\sum_{i=1}^K f(z_i+u)\big|^{2\xi}\, du\Big]\\
\label{eq:moments1}
&=c\,\int_{[0,\cR]^d}\bbE\Big[\big|\sum_{i=1}^K f(z_i+u)\big|^{2\xi}\Big]\, du.
\end{align}
For fixed $u\in[0,\cR]^d$ the sequence $\big(f(z_i+u)\big)_{i=1}^K$ has mean zero and is independent by Assumption~\ref{ass:finite range}. So we have by Lemma~\ref{lemma:indep rvs} and shift-invariance of the environment,
\begin{align}
\nonumber
\bbE\Big[\big|\sum_{i=1}^K f(z_i+u)\big|^{2\xi}\Big]&\leq c_{13}\,\max\Big\{\sum_{i=1}^K\bbE\big[\big|f(z_i+u)\big|^{2\xi}\big],\, \big(\sum_{i=1}^K\bbE\big[\big|f(z_i+u)\big|^{2}\big]\big)^{\xi}\Big\}\\
\label{eq:moments2}
&\leq c\, K^\xi.
\end{align}
Combining \eqref{eq:moments1} and \eqref{eq:moments2} gives the result.
\end{proof}

\begin{proposition}
\label{prop:ergodic terms bnd}
Let $\xi>d$ and assume $M_2(2\xi p, 2\xi q)<\infty$. For $\bbP$-a.e.~$\om$, there exists $N_4(\om)\in\bbN$ such that for all $r>0$ and $x\in\bbR^d$ with $N_4(\om)<r\leq 4\,d(0,x)$, for any sequence $y_0,\dots,y_k$ where $y_0=0$, $y_k=x$ and $y_j\in B_{x_j}$ for $1\leq j \leq k-1$, we have $c_{15}(d,\cR,\xi)>0$ such that
\begin{equation}
\label{eq:ergodic bound}
\sum_{j=0}^{k-1}\Big(1\vee \norm{\La^\om}_{p,B(y_j,\sqrt{s})}\Big)\Big(1\vee\norm{\la^\om}_{q,B(y_j,\sqrt{s})}\Big)\leq c_{15}\, k.
\end{equation}
Furthermore, we have the following estimate on $N_4(\om)$, there exists $c_{16}(d,\cR,\xi)>0$ such that
\begin{equation}
\label{eq:N decay2}
\bbP(N_4(\om)>n)\leq c_{16}\,n^{2-d(\xi-1)}\qquad \forall\, n\in\bbN.
\end{equation}
\end{proposition}

\begin{proof}
Let $x$ and $r$ be as in the statement and denote $z=\floor{x}\in\bbZ^d$, $r_0=\ceil{r}\in\bbZ$. We will work with these discrete approximations of the variables $x$ and $r$ in order to apply countable union bounds and the Borel-Cantelli lemma. Note that $x\in C_z:=z+[0,1]^d$ and $r\in I_{r_0}:=[r_0-1,r_0]$. Assuming w.l.o.g.~that $r>1$ and $d(0,x)>d$ we have $r\simeq r_0$ and $\abs{x}\simeq\abs{z}$. We define a region that covers the union of balls of interest
$$\bigcup_{j=0}^k B_{y_j}\subset R_{z,r_0}:=\Big\{\tau z+\big[-2r_0,2r_0\big]^d\, : \, \tau\in\big[0,2\big]\Big\}.$$
This region has volume $|R_{z,r_0}|\leq c\, r_0^{d-1}\abs{z}\leq c\,r^d k$ and can be covered by at most $K\leq c\,r_0^{d-1}\abs{z}/\cR^d$ non-intersecting balls of radius $\cR$ in the maximal norm. Also there exists $c_{17}(d)$ such that for all $w\in\bbR^d$, $|\{j\in\{0,\dots,k\}\, : \, w \in B_{y_j}\}|\leq c_{17}$, therefore
\begin{align}
\nonumber
\sum_{j=0}^{k-1} \norm{\La^\om}_{p,B(y_j,\sqrt{s})}^p&\leq c_{17}\,r^{-d}\int_{\bigcup_{j=0}^k B_{y_j}}\La^\om(u)^p\, du\leq c\,r^{-d}\int_{R_{z,r_0}}\La^\om(u)^p\, du\\
\nonumber
& \leq c\,r^{-d} \abs{R_{z,r_0}}\bbE\big[\La^\om(0)^p\big]+c\,r^{-d}\int_{R_{z,r_0}} \De\La_p^\om(u)\, du\\
\label{eq:bound}
&\leq c\, k+c\,r^{-d}\int_{R_{z,r_0}} \De\La_p^\om(u) \, du.
\end{align}
By Markov's inequality and Proposition~\ref{prop:moment bound} we have
\begin{align}
\nonumber
\bbP\Big(\int_{R_{z,r_0}} \De\La_p^\om(u)\, du>kr^d \Big)&\leq \bbP\Big(\big|\int_{R_{z,r_0}} \De\La_p^\om(u)\, du\big| >c\abs{z}r_0^{d-1} \Big)\\
\nonumber
&\leq c\,\bbE\Big[\big| \int_{R_{z,r_0}} \De\La_p^\om(u)\, du\big|^{2\xi}\Big]/\big(\abs{z}r_0^{d-1}\big)^{2\xi}\\
\label{eq:bound1}
&\leq c\,\big(\abs{z}r_0^{d-1}\big)^{-\xi}.
\end{align}
Now let $\rho, \, l\in\bbN$ with $\rho\leq l$. By \eqref{eq:bound1} and a union bound, summing over $\{z\in\bbZ^d\, : \, \abs{z}=l\}$ and $r_0\geq \rho$,
\begin{multline}
\label{eq:bound2}
\bbP\Big(\exists\, z\in\bbZ^d,\, r_0\in\bbN\, : \,\abs{z}=l, r_0\in[\rho, 4\abs{z}],  \int_{R_{z,r_0}} \Delta\La_p^\om(u)\, du>kr^d \Big)\\
\leq c\, l^{d-1-\xi}\rho^{-\xi(d-1)+1}.
\end{multline}
Now consider the event
$$E_\rho:= \Big\{\exists\, z\in\bbZ^d,\, r_0\in\bbN\, : \,\abs{z}\geq\rho, r_0\in[\rho, 4\abs{z}],  \int_{R_{z,r_0}} \Delta\La_p^\om(u)\, du>kr^d\Big\}.$$
Since $\xi>d$, we can take a countable union bound over $l$ in \eqref{eq:bound2} to obtain
\begin{equation}
\label{eq:rho event}
\bbP\big(E_\rho)\leq c\, \rho^{d(1-\xi)+1}.
\end{equation}
Also $d(1-\xi)+1<-1$ so by the Borel-Cantelli lemma there exists $\tilde{N}(\om)\in\bbN$ such that for all $z\in\bbZ^d,\, r_0\in\bbN$ with $\tilde{N}(\om)<r_0<4\abs{z}$ we have
$$\int_{R_{z,r_0}} \Delta\La_p^\om(u)\, du\leq kr^d.$$
Together with \eqref{eq:bound}, this implies the existence of $N_4(\om)\in\bbN$ such that for all $x\in\bbR^d$ and $r>1$ with $N_4(\om)<r\leq 4\,d(0,x)$ we have for $y_0,\dots, y_k$ defined as in the statement,
\begin{equation}
\label{eq:initial La bound}
\sum_{j=0}^{k-1} \norm{\La^\om}_{p,B(y_j,\sqrt{s})}^p\leq c\,k.
\end{equation}
By the exact same reasoning, one can show the corresponding inequality for $\la^\om$. Moreover by H\"older's inequality,
\begin{multline}
\label{eq:holder step}
\sum_{j=0}^{k-1}\Big(1\vee \norm{\La^\om}_{p,B(y_j,\sqrt{s})}\Big)\Big(1\vee\norm{\la^\om}_{q,B(y_j,\sqrt{s})}\Big)\\
\leq k^{1-\frac{1}{p}-\frac{1}{q}}\Big(\sum_{j=0}^{k-1}\big(1\vee \norm{\La^\om}_{p,B(y_j,\sqrt{s})}^p\big)\Big)^{\frac{1}{p}}\Big(\sum_{j=0}^{k-1}\big(1\vee\norm{\la^\om}^q_{q,B(y_j,\sqrt{s})}\big)\Big)^\frac{1}{q}.
\end{multline}
This together with \eqref{eq:initial La bound} and the equivalent bound for $\la^\om$ gives the result. The stated decay of $N_4(\om)$ follows by taking a union bound over $\rho$ in \eqref{eq:rho event}.
\end{proof}

\begin{corollary}
\label{cor:kappa terms bound}
Let $\xi>d$ and assume $M_2(2\xi\kappa p, 2\xi\kappa q)<\infty$. In the same setting as Proposition~\ref{prop:ergodic terms bnd} there exists $N_5(\om)\in\bbN$ with decay as in \eqref{eq:N decay2} and $c_{18}(d,p,q,\cR,\xi)>0$ such that $\bbP$-a.s.~ for all $r>0$, $x\in\bbR^d$ with $N_5(\om)<r\leq 4\,d(0,x)$ we have
\begin{equation}
\label{eq:kappa ergodic terms bnd}
\sum_{j=0}^{k-1}\Big(1\vee \norm{\La^\om}_{p,B(y_j,\sqrt{s})}\Big)^\kappa\Big(1\vee\norm{\la^\om}_{q,B(y_j,\sqrt{s})}\Big)^\kappa\leq c_{18}\, k.
\end{equation}
\end{corollary}
\begin{proof}
By Jensen's inequality $\norm{\La^\om}_{p,B(y_j,\sqrt{s})}^\kappa\leq \norm{(\La^\om)^\kappa}_{p,B(y_j,\sqrt{s})}$ and similarly for the $\la^\om$ terms. Then proceed as for Proposition~\ref{prop:ergodic terms bnd} to prove the result, with $\La^\om$ replaced by $(\La^\om)^\kappa$.
\end{proof}

\begin{proof}[Proof of Theorem~\ref{thm:lower estimate}]
By shift-invariance of the environment it suffices to prove the estimate for $p_\La^\om(t,0,x)$. Fix $\xi>d$ and for the moment assumption $M_2(p_0,q_0)<\infty$ choose $p_0=2\xi\kappa p$, $q_0=2\xi\kappa q$, in order to apply Corollary~\ref{cor:kappa terms bound}. Let $N_3^\om(0):=N_1^\om(0)^2\vee N_4^\om \vee N_5^\om$ and assume as in the statement that $t\geq N_3^\om(0)\big(1\vee d(0,x)\big)$. We split the proof into two cases.

Firstly in the case $\abs{x}^2/t<1/4$ we have $x \in B\big(0,\sqrt{t}/2\big)$ so we may apply the near-diagonal lower estimate of Proposition~\ref{prop:ndle},
$$p_{\La}^\om(t,0,x)\geq \frac{t^{-d/2}}{C_{\textrm{PH}}\big(\norm{\Lambda^\om}_{p,B(0,\sqrt{t})},\norm{\lambda^\om}_{q,B(0,\sqrt{t})}\big)}.$$
Since $\sqrt{t}\geq N_1^\om(0)$, recalling the form of $C_{\textrm{PH}}$ we apply the ergodic theorem to bound
\begin{equation}
C_{\textrm{PH}}\big(\norm{\Lambda^\om}_{p,B(0,\sqrt{t})},\norm{\lambda^\om}_{q,B(0,\sqrt{t})}\big) \leq c_{11}\exp\big(c\,\big((1\vee\bar{\Lambda}_p)(1\vee\bar{\lambda}_q)\big)^\kappa\big).
\end{equation}
Therefore,
$$p_{\La}^\om(t,0,x)\geq c\, t^{-d/2}.$$

Secondly, consider the case $\abs{x}^2/t\geq 1/4$.  Since $\La^\om$ and $\la^\om$ are locally bounded, it follows from the semigroup property that for any $0<\tau<t$,
\begin{equation}
p_\La^\om(t,0,x)=\int_{\bbR^d}p_\La^\om(\tau,0,u)p_\La^\om(t-\tau,u,x)\La^\om(u)\, du.
\end{equation}
We will employ the chaining argument over the sequence of balls introduced below Proposition~\ref{prop:ndle}, set $r=t/\abs{x}\geq N_3^\om(0)$ which gives $s=t/k$. Iterating the above relation $k-1$ times gives
\begin{multline}
\nonumber
p_\La^\om(t,0,x)\geq\\
 \int_{B_{x_1}}\dots\int_{B_{x_{k-1}}} p_\La^\om(s,0,y_1)\dots p_\La^\om(s,y_{k-1},x)\La^\om(y_1)\dots\La^\om(y_{k-1})\,dy_1\dots dy_{k-1}.
\end{multline}
We have by Proposition~\ref{prop:ndle}, for all $y_j\in B_{x_j}$,
\begin{align}
\nonumber
\prod_{j=0}^{k-1} p_\La^\om(s,y_j,y_{j+1})&\geq \frac{c\,s^{-dk/2}}{\exp\Big(c\sum_{j=0}^{k-1}\Big(\big(1\vee\norm{\Lambda^\om}_{p,B(y_j,\sqrt{s})}\big)\big(1\vee\norm{\lambda^\om}_{q,B(y_j,\sqrt{s})}\big)\Big)^\kappa\Big)}\\
&\geq \frac{c\, s^{-dk/2}}{\exp(c\,k)},
\end{align}
where the second step is due to Corollary~\ref{cor:kappa terms bound}. Therefore,
\begin{align}
\nonumber
p_\La^\om(t,0,x)& \geq \frac{c\, s^{-dk/2} \prod_{j=1}^{k-1}|B_{x_j}|\norm{\La^\om}_{1,B_{x_j}}}{\exp(c\, k)}\\
\label{eq:lower2}
&\geq \frac{c\, r^{-dk} r^{d(k-1)}\prod_{j=1}^{k-1}\norm{\La^\om}_{1,B_{x_j}}}{c^k}.
\end{align}
To bound the remaining stochastic term in the numerator we apply the harmonic-geometric mean inequality,
\begin{equation}
\Big(\prod_{j=1}^{k-1}\norm{\La^\om}_{1,B_{x_j}}\Big)^\frac{1}{k-1} \geq \frac{k-1}{\sum_{j=1}^{k-1}\norm{\La^\om}_{1,B_{x_j}}^{-1}}\geq \frac{c(k-1)}{\sum_{j=1}^{k-1}\norm{\la^\om}_{1,B_{x_j}}}.
\end{equation}
%
%
Since $r>N_4^\om$, it follows from Proposition~\ref{prop:ergodic terms bnd} with the choice $y_j=x_j$, that $\sum_{j=1}^{k-1}\norm{\la^\om}_{1,B_{x_j}}\leq c\,k$. Therefore,
\begin{equation}
\label{eq:lower3}
\prod_{j=1}^{k-1}\norm{\La^\om}_{1,B_{x_j}}\geq c^k.
\end{equation}
Combining \eqref{eq:lower2} and \eqref{eq:lower3} gives for some $c_{19}>0$, $c_{20}\in(0,1)$,
\begin{equation}
p_\La^\om(t,0,x)\geq  c_{19}\,r^{-d}\,c_{20}^k.
\end{equation}
Finally, since $\frac{\abs{x}^2}{t}\geq\frac{1}{4}$ we have $r\leq 2\,t^{1/2}$. Also $k\simeq \frac{\abs{x}}{r}=\frac{\abs{x}^2}{t}$ so we arrive at
\begin{equation}
p_\La^\om(t,0,x)\geq c_2\, t^{-\frac{d}{2}}\exp\Big(-\frac{c_3\, d(0,x)^2}{t}\Big),
\end{equation}
which completes the proof.
\end{proof}

\section{Green's Function Scaling Limit}
\label{section:green kernel}

We shall now prove the Green's function scaling limit in Theorem~\ref{thm:green kernel}. The strategy is to apply the local limit theorem \cite[Theorem~1.1]{CD15} then control remainder terms using the off-diagonal estimate of Theorem~\ref{thm:off-diagonal estimate} and the long range bound established below in Proposition~\ref{prop:long range}. Throughout this section suppose Assumption~\ref{ass:ergodicity} holds and let $d\geq 3$ so that the Green's function exists. Also let $p,\,q\in(1,\infty]$ satisfying $\frac{1}{p}+\frac{1}{q}<\frac{2}{d}$.

Herein, since the Green's function is independent of the choice of speed measure, we specify the case $\th^\om\equiv \Lambda^\om$. This choice is analogous to the constant speed random walk in the random conductance model setting cf.~for example \cite{harnack}, for which the random walk's jump rate is constant and independent of its position. The benefit of this choice of speed measure is that the intrinsic metric may be bounded in terms of the Euclidean metric, leading to particularly amenable off-diagonal bounds. By the definition of $\La^\om$, the intrinsic metric satisfies
\begin{align}
\nonumber
d^\om_\La(x,y) &:=\sup\Big\lbrace \phi(y)-\phi(x) \, :\, \phi \in C^1(\bbR^d),\, h^\om(\phi)^2= \sup_{z\in\bbR^d}\frac{(\grad\phi\cdot a^\om\grad\phi)(z)}{\La^\om(z)}\leq 1\Big\rbrace\\
\label{eq:metric comparable}
& \geq \sup\Big\lbrace \phi(y)-\phi(x) \, :\, \phi \in C^1(\bbR^d),\, \norm{\grad\phi}_\infty \leq 1\Big\rbrace=d(x,y).
\end{align}
The final equivalence here is due to the fact that the Euclidean metric is the Riemannian metric corresponding to the identity matrix.

\begin{corollary}
\label{cor:euclidean off-diagonal}
Suppose $M_2(p,q)<\infty$. For $\bbP$-a.e.~$\om$, there exist $N_6^\om(x)>0$ and $c_{21}(d,p,q),\, c_{22}(d,p,q)>0$ such that for all $x,\,y\in\bbR^d$, $\sqrt{t}>N^\om_6(x)$,
\begin{equation}
\label{eq:euclidean off-diagonal}
p_\La^\om(t,x,y)\leq c_{21}\, t^{-\frac{d}{2}} \exp\Big(-c_{22}\frac{d(x,y)^2}{t}\Big).
\end{equation}
\end{corollary}
\begin{proof}
This follows by the exact reasoning of Theorem~\ref{thm:off-diagonal estimate}, noting that $h^\om(\phi)^2\leq \norm{\grad\phi}_\infty^2$. Since the Euclidean metric is trivially strictly local, the justification involving Assumption~\ref{ass:continuity} is no longer required.
\end{proof}
Whilst the above off-diagonal estimate provides optimal bounds on the heat kernel for large enough time $t$, it is clear that to control the convergence in \eqref{eq:green scaling} we also require a bound on the rescaled heat kernel that holds for small $t>0$. We obtain this from the following long range bound, derived in a similar fashion to results in the graph setting such as \cite[Theorem~10]{Dav93}. Interestingly, we obtain stronger decay in the present diffusion setting than for the aforementioned random walks on graphs \cite{Dav93}, \cite{Pan93}, where a logarithm appears in the exponent. See also \cite[Theorem~1.6(ii)]{ADS16} for the degenerate environment.
\begin{proposition} 
\label{prop:long range}
Suppose $M_2(p,q)<\infty$. For $\bbP$-a.e.~$\om$, there exist $c_{23}>0$ and $N_7(\om)>0$ such that for all $n\geq N_7(\om)$, $t\geq 0$ and $x\in\bbR^d$ with $\abs{x}\leq 2$ we have
\begin{equation}
\label{eq:long range}
p_{\La}^\om(t,0,nx)\leq c_{23}\, n^d\exp\Big(-\frac{n^2\abs{x}^2}{2\,t}\Big).
\end{equation}
\end{proposition}
\begin{proof}
Firstly note that by Lemma~\ref{lemma:cauchy estimate}, for any $f\in L^2(\bbR^d,\th^\om dx)$ and suitable $\psi$,
\begin{equation}
\norm{\me^\psi P_t f}_{2,\Lambda}^2\leq \me^{h^\om(\psi)^2t}\norm{\me^\psi f}_{2,\La}^2.
\end{equation}
By the local boundedness in Assumption~\ref{ass:ergodicity}, this implies the pointwise estimate
\begin{equation}
\me^{2\,\psi(x)}p_\La^\om(t,x,y)^2\La^\om(y)^2\La^\om(x)\leq \me^{h^\om(\psi)^2 t+2\,\psi(y)}\La^\om(y),
\end{equation}
for all $t\geq 0$ and $x,\, y \in \bbR^d$. Rearranging,
\begin{align*}
p_\La^\om(t,x,y)&\leq \La^\om(y)^{-1/2}\La^\om(x)^{-1/2}\exp\big(h^\om(\psi)^2t/2+\psi(y)-\psi(x)\big)\\
& \leq \La^\om(y)^{-1/2}\La^\om(x)^{-1/2}\exp\big(\norm{\nabla \psi}_\infty^2t/2+\psi(y)-\psi(x)\big).
\end{align*}
Arguing as in Section~\ref{subsection:proof of UE} but with the Euclidean metric gives
\begin{equation}
\nonumber
p_\La^\om(t,x,y) \leq \La^\om(y)^{-1/2}\La^\om(x)^{-1/2}\exp\Big(-\frac{d(x,y)^2}{2\,t}\Big).
\end{equation}
Now set $x=0$ and re-label $y=nx$ with $\abs{x}\leq 2$,
\begin{equation}
\label{eq:long range1}
p_\La^\om(t,0,nx)\leq \La^\om(0)^{-1/2}\La^\om(nx)^{-1/2}\exp\Big(-\frac{n^2\abs{x}^2}{2\,t}\Big).
\end{equation}
Appealing to ergodicity of the environment, by Assumption~\ref{ass:ergodicity} and the moment condition, $\bbP$-a.s.~there exists $N_7(\om)>0$ such that for all $n \geq N_7(\om)$,
\begin{align}
\nonumber
\La^\om(nx)^{-1}& \leq \la^\om(nx)^{-1} \leq \int_{B(0,2n)} \la^\om(u)^{-1}\, du \\
\nonumber
& \leq c \, n^d \norm{ 1/\la^\om}_{1, B(0,2n)} \leq c\, n^d \norm{1/\la^\om}_{q,B(0,2n)}\quad \text{(Jensen's inequality)} \\
\label{eq:la bound1}
& \leq c\, n^d\, \bbE\big[\la^\om(0)^{-q}\big]^{1/q}.
\end{align}
Similarly,
\begin{equation}
\label{eq:la bound2}
\La^\om(0)^{-1}\leq c\, n^d\, \bbE\big[\la^\om(0)^{-q}\big]^{1/q} \text{\quad for all }n\geq N_7(\om).
\end{equation}
Substituting \eqref{eq:la bound1} and \eqref{eq:la bound2} into \eqref{eq:long range1} gives the result.
\end{proof}

\begin{proof}[Proof of Theorem~\ref{thm:green kernel}] By shift-invariance of the environment it suffices to prove the result for $x_0=0$. For simplicity we set $r_1=1$, $r_2=2$, and in a slight abuse of notation we write $k_t^\Si(x)=k_t^\Si(0,x)$. For $1\leq \abs{x}\leq 2$, $T_1,\, T_2>0$ and $n>0$ we have
\begin{align}
\nonumber
&\abs{ n^{d-2}g^\omega(0,nx)-a \, g_{BM}(0,x)}=\Big| n^{d}\int_0^\infty p_\La^\om(n^2t,0,nx)\, dt-a \, \int_0^\infty k_t^\Si(x)\,dt\Big| \\
\nonumber
\label{eq:split terms}
& \leq n^d \int_0^{T_1} p_\La^\om(n^2t, 0, nx)\, dt+a\int_0^{T_1}k_t^\Si(x)\, dt+\int_{T_1}^{T_2}\big| n^d p_\La^\om(n^2t,0,nx)-a\, k_t^\Si(x)\big| \, dt\\
&\qquad + n^d \int_{T_2}^\infty p_\La^\om(n^2t, 0, nx)\, dt+a\int_{T_2}^\infty k_t^\Si(x)\, dt.
\end{align}
In controlling these terms we first employ the main result of this paper; the off-diagonal estimate in Corollary~\ref{cor:euclidean off-diagonal} gives
\begin{equation}
\label{eq:p int upper}
n^d \int_{T_2}^\infty p_\La^\om(n^2t, 0, nx)\, dt \leq c_{21}\, \int_{T_2}^\infty t^{-d/2}\me^{-c_{22}/t}\, dt,
\end{equation}
provided $n>\sqrt{N_6^\om(0)/T_2}$. Similarly, for the Gaussian heat kernel there exists $c>0$ such that for all $t\geq 0$ and $1 \leq \abs{x} \leq 2$,
\begin{equation}
\label{eq:k int}
k_t^\Si(x)\leq c\, t^{-d/2}\me^{-c/t}.
\end{equation}
For the first term in \eqref{eq:split terms} we apply both the off-diagonal estimate and the long range bound of Proposition~\ref{prop:long range}. Provided $n>N_7(\om)\vee \big(N_6^\om(0)/\sqrt{T_1}\big)$, we have
\begin{align}
\nonumber
& n^d \int_0^{T_1} p_\La^\om(n^2t, 0, nx)\, dt \leq n^d \int_0^{\frac{N_6^\om(0)^2}{n^2}} p_\La^\om(n^2t, 0, nx)\, dt + n^d \int_{\frac{N_6^\om(0)^2}{n^2}}^{T_1} p_\La^\om(n^2t, 0, nx)\, dt\\
\nonumber
&\leq\, n^d \int_0^{\frac{N_6^\om(0)^2}{n^2}} c\, n^d\me^{-\frac{1}{2t}} \, dt + n^d \int_{\frac{N_6^\om(0)^2}{n^2}}^{T_1} c\, n^{-d}t^{-\frac{d}{2}} \me^{-\frac{c}{t}} \, dt\quad \text{(by \eqref{eq:long range} and \eqref{eq:euclidean off-diagonal} resp.)} \\
\label{eq:p int lower}
&\leq \, c\, N_6^\om(0)^2n^{2d-2}\exp\big(-n^2/\big(2N_6^\om(0)^2\big)\big)+c\, \int_0^{T_1}t^{-\frac{d}{2}} \me^{-\frac{c}{t}} \, dt.
\end{align}
Let $\eps >0$. Combining the above we have that for suitably large $n$,
\begin{multline}
\nonumber
\abs{ n^{d-2}g^\omega(0,nx)-a \, g_{BM}(0,x)}\leq c\,(1+a)\Big(\int_0^{T_1}t^{-d/2}\me^{-c/t} \, dt+\int_{T_2}^\infty t^{-d/2}\me^{-c/t}\, dt\\
+ N_6^\om(0)^2n^{2d-2}\exp\big(-n^2/\big(2N_6^\om(0)^2\big)\big)+\int_{T_1}^{T_2}\big| n^d p_\La^\om(n^2t,0,nx)-a\, k_t^\Si(x)\big| \, dt\Big).
\end{multline}
Now, $t^{-d/2}\me^{-c/t}$ is integrable on $(0,\infty)$ so we may fix $T_1$, $T_2$ such that
$$\int_0^{T_1}t^{-d/2}\me^{-c_2/t} \, dt+\int_{T_2}^\infty t^{-d/2}\me^{-c_2/t}\, dt<\eps.$$
For large enough $n$,
$$N_6^\om(0)^2n^{2d-2}\exp\big(-n^2/\big(2N_6^\om(0)^2\big)\big)<\eps.$$
Furthermore, by the local limit theorem \cite[Theorem 1.1]{CD15},
$$\int_{T_1}^{T_2}\big| n^d p_\La^\om(n^2t,0,nx)-a\, k_t^\Si(x)\big| \, dt<\eps, $$
for large enough $n$, uniformly over $x\in A$. This gives the claim.
\end{proof}

\noindent\textit{Acknowledgement.} With thanks to Sebastian Andres for valuable discussions and guidance on this project.
\bibliographystyle{abbrv}
\bibliography{refs}

\end{document}